\documentclass[reqno,11pt]{amsart}

\usepackage[dvips,pdftex]{epsfig}
\usepackage{amssymb}
\usepackage{amsmath}
\usepackage{amscd}
\usepackage{amsthm}
\usepackage[mathscr]{eucal}
\usepackage[all]{xy}
\usepackage{color}

\newtheorem{theorem}{Theorem}
\newtheorem{lemma}{Lemma}
\newtheorem{corollary}{Corollary}

\theoremstyle{remark}
\newtheorem{remark}{Remark}

\theoremstyle{definition}
\newtheorem{definition}{Definition}

\DeclareMathOperator{\elm}{elm}

\DeclareMathOperator{\Jac}{Jac}

\DeclareMathOperator{\Ab}{Ab}
\DeclareMathOperator{\dAb}{dAb}
\DeclareMathOperator{\bl}{Bl}

\DeclareMathOperator{\mult}{mult}

\DeclareMathOperator{\Hom}{Hom}

\DeclareMathOperator{\HH}{H}

\DeclareMathOperator{\TT}{T}

\DeclareMathOperator{\WW}{W}

\DeclareMathOperator{\pic}{Pic}

\begin{document}

\title[Hyperelliptic curves and Hitchin tangential covers]{Hyperelliptic curves and Hitchin tangential covers}
\author[Taejung Kim]{Taejung Kim}

\address{Korea Institute for Advanced Study\\
207-43 Cheongyangri-dong\\
Seoul 130-722, Korea}

\thanks{The author would like to express his sincere gratefulness to the Korea Institute for Advanced Study for providing him a hospital environment
while preparing this manuscript and to William Goldman, Serguei Novikov, and Niranjan Ramachandran for giving him their deep insights, lively discussions, and valuable suggestions about this paper.}


\date{\today}
\email{tjkim@kias.re.kr}

\begin{abstract}
In \cite{kim11} we have generalized a tangency condition in the Treibich-Verdier theory \cite{trei89,tv90,trei97} about elliptic solitons to a Hitchin system. As an application of this generalization, we will define, so-called, Hitchin hyperelliptic tangential covers and show the finiteness of them among all Hitchin tangential covers.
\end{abstract}
\maketitle

\section{Introduction}

In 1990, Treibich and Verdier completely characterized, so-called, elliptic solitons. In their theory, what was essential was the concept of a ``tangential cover''. The new concept indeed makes it possible to classify analytic objects by means of algebraic tools. As a spectacular application of the concept, they showed that there are only finitely many hyperelliptic tangential covers over a fixed elliptic curve. In \cite{kim11} we generalized the concept of a ``tangential cover'' to a Hitchin system (see \cite{don96,hit87,kim11} for the details of the Hitchin theory). Loosely speaking, the generalized tangency condition in \cite{kim11} provides a way of taking a particular family $\mathscr{HT}(n,g,\mathfrak{S})$ of Hitchin spectral covers \emph{algebraically}. The main idea of \cite{kim11} is as follows: We construct a particular ruled surface $\mathfrak{S}$ by elementary transformations from $\mathfrak{R}\times\mathbb{P}^1$ where the genus of a Riemann surface $\mathfrak{R}$ is $g>1$ and we consider a family of Hitchin spectral curves which can be realized as divisors on $\mathfrak{S}$. We observe that those Hitchin spectral curves in the surface $\mathfrak{S}$ will have some special property similar to elliptic solitons, i.e., tangential covers over an elliptic curve, in the Treibich-Verdier theory. Hence, we will call them Hitchin tangential covers as a generalization of the tangential covers. We denote the space of all Hitchin tangential covers $\pi:\widehat{\mathfrak{R}}\to\mathfrak{R}$ of degree $n$ by $\mathscr{HT}(n,g,\mathfrak{S})$. In \cite{kim11}, we proved for $n>2$
\begin{equation}\label{dimeq}
\dim_\mathbb{C}\mathscr{HT}(n,g,\mathfrak{S})=(n^2-1)(g-1)-1.
\end{equation}

As a sequel of \cite{kim11}, we will continue to investigate the generalized tangency condition and exhibit one application of this analysis in this paper. That is, we will prove the finiteness of Hitchin hyperelliptic tangential covers over a hyperelliptic curve. However, in order to prove the claim, we may encounter one difficulty which comes from a way of constructing $\mathfrak{S}$ in \cite{kim11}. In \cite{tv90}, Treibich and Verdier construct a ruled surface $S$ over an elliptic curve $X$ where a natural involution of the elliptic curve $X$ is canonically extended to a natural involution of $S$. It was the key idea of proving the finiteness of hyperelliptic tangential covers in \cite{tv90}. Hence, we have to provide a similar construction of a ruled surface $\mathfrak{S}$ over a hyperelliptic curve $(\mathfrak{R},\tau)$ in a way of having a naturally induced involution $\tau$ on $\mathfrak{S}$ from $\tau$ on $\mathfrak{R}$ analogous to the Treibich-Verdier construction. Once we construct the ruled surface $\mathfrak{S}$ in this way, the rest of the proofs of the finiteness naturally follows from mimicking the techniques in the Treibich-Verdier theory \cite{tv90}. Hence, the general theme of the Treibich-Verdier theory will be repeated. Consequently, we should mention that the enormous influence of paper \cite{tv90} to this paper should be credited and we would like to express our indebtedness of insights and guidance from \cite{tv90} during writing this paper. We hope that this paper would manifest and develop the versatility and generality of the core algebraic procedure in the Treibich-Verdier theory by illustrating its adaptability to a Hitchin dynamical system in addition to the results obtained in this paper about the Hitchin system.

The structure of this paper is as follows: In Section~\ref{unisec1}, we will formulate another construction of a ruled surface which we have constructed in \cite{kim11} where Hitchin spectral curves can be defined as divisors. From this new point of view of construction, when a base curve $\mathfrak{R}$ is a hyperelliptic curve, the associated ruled surface $\mathfrak{S}$ will obtain a natural involution which plays an important role in describing Hitchin hyperelliptic tangential covers. Hence, we can observe that the surface $\mathfrak{S}$ is the right generalization of the surface $S$ in the elliptic soliton theory \cite{tv90}. In Section~\ref{tansec3}, we will characterize Hitchin hyperelliptic tangential covers over a hyperelliptic curve by adapting the framework of Treibich-Verdier theory. In Section~\ref{secfin4}, we will show the finiteness of Hitchin hyperelliptic tangential covers over a hyperelliptic curve by calculating the dimension of a linear system explicitly.

\section{The universal affine bundle and the ruled surface}\label{unisec1}

In \cite{kim11}, we constructed a ruled surface for a Hitchin system to generalize the Treibich-Verdier theory for elliptic solitons. In this section we will give another construction of the same ruled surface in \cite{kim11}. The equivalence of those two constructions is the essential part of characterization of the finiteness of hyperelliptic Hitchin tangential covers in the moduli of all the Hitchin tangential covers. For the complete backgrounds of the construction in this section, we refer \cite{gol08, har77, maru70}:

Let $\mathfrak{R}$ be a compact Riemann surface of genus $g>0$. From \cite{gol08}, we know that $\Hom(\pi_1(\mathfrak{R}),\mathbb{C}^\ast)/\mathbb{C}^\ast$, which is called a Betti groupoid, is holomorphically isomorphic to $\HH^{1}_{DR}(\mathfrak{R},\mathbb{C})/\HH^1(\mathfrak{R},\mathbb{Z})\cong\HH^1(\mathfrak{R},\mathbb{C^\ast})\cong(\mathbb{C}^\ast)^{2g}$, so-called a de Rham groupoid, which is the moduli space of flat $\mathbb{C}^\ast$-connections over $\mathfrak{R}$. On the other hands, a moduli space of Higgs bundles, $\TT^\ast\Jac(\mathfrak{R})\cong\Jac\mathfrak{R}\times\HH^{1,0}(\mathfrak{R})$, which is called a Dolbeault groupoid, is not biholomorphic to the other two spaces, since it is not Stein. Nevertheless, all three spaces are diffeomorphic to each other. In particular, the correspondence between a de Rham groupoid and a Dolbeault groupoid is given by,
$$D_0+[\eta]\to \bigg[\big((i\Im\eta)^{0,1},(\Re\eta)^{1,0}\big)\bigg](p.34\text{ in \cite{gol08}})$$
where $\eta$ is a harmonic form and $D_0$ is the reference flat connection. Using this correspondence, we have a short exact sequence of commutative groups
\begin{equation}\label{grse1}
0\to(\mathbb{G}_{a})^{g}\to(\mathbb{C}^\ast)^{2g}\to \Jac(\mathfrak{R})\to 0
\end{equation}
where $\mathbb{G}_a$ is the affine group of translations. So we have constructed an affine bundle $\widehat{\Delta}$ over $\Jac(\mathfrak{R})$ with fiber $(\mathbb{G}_{a})^{g}$. Note that this is a generalization of a sequence in {\em p.445} in \cite{tv90}. It is well-known that there is a universal affine embedding of $\widehat{\Delta}$ to a holomorphic bundle of rank $g+1$. Using the group structures of \eqref{grse1}, we may see that the structure group $\{G_{ij}\}$ of transition functions for this case is given by
$$G_{ij}=\begin{pmatrix}
  1& 0 & \cdots &z'_{1j}-z'_{1i}  \\
  0& 1 & \cdots & z'_{2j}-z'_{2i} \\
  \vdots  & \vdots  & \ddots & \vdots  \\
  0 & 0 & \cdots & z'_{gj}-z'_{gi}\\
 0 & 0 & \cdots & 1
 \end{pmatrix}\in\mathbf{SL}(g+1,\mathbb{C})\subset\mathbf{GL}(g+1,\mathbb{C}).$$
Here $(z'_{1,i},\dots,z'_{g,i})$ is a local coordinate of $U'_i$ where $\{U'_i\}$ is an open cover of  $\Jac(\mathfrak{R})$. The affine embedding transformation is given by
$$\xymatrix{(w_1,\dots,w_n,1)\ar[r]^(0.3){G_{ij}}& (w_1+(z'_{1j}-z'_{1i}),\dots,w_n+(z'_{gj}-z'_{gi}),1).}$$

Since $\TT^\ast\Jac(\mathfrak{R})\cong\Jac\mathfrak{R}\times\HH^{1,0}(\mathfrak{R})$ is a trivial bundle, $\TT\Jac(\mathfrak{R})\cong\Jac\mathfrak{R}\times\HH^{0,1}(\mathfrak{R})$ is a trivial bundle. Consider
$$\xymatrix{
\TT\mathfrak{R}\ar@{-->}^{\psi}[dr]\ar[r]^(0.4){\dAb}&\TT\Jac(\mathfrak{R})\ar[d]^{\phi_1}&\\
&\widehat{\Delta}\ar[r]^{\phi_2}&(\mathbb{C}^\ast)^{2g}}$$
Here $\phi_1$ is a diffeomorphism which is the identity on $\Jac(\mathfrak{R})$ and $\phi_2$ is a biholomorphism. From this diagram, we see that the tangent bundle $\TT\mathfrak{R}$ has an induced affine structure. Note that this procedure is nothing but giving a new complex structure to $\TT\Jac(\mathfrak{R})$ induced from $(\mathbb{C}^\ast)^{2g}$. Since the tangent bundle $\TT\mathfrak{R}$ is naturally embedded into $\TT\Jac(\mathfrak{R})$ by the Abel map, we may induce a new complex structure on the tangent bundle $\TT\mathfrak{R}$. Let us denote $\Delta:=\phi_1\circ\dAb(\TT\mathfrak{R})$. Hence, $\Delta$ is an affine bundle over $\mathfrak{R}$, i.e., a principal $\mathbb{G}_a$-bundle over $\mathfrak{R}$,  with this complex structure. Again, by the universal affine embedding of $\Delta$, we have a holomorphic vector bundle $\WW$ of rank $2$ with the structure group $\{G_{ij}\}$ of transition functions given by
$$G_{ij}=\begin{pmatrix}
1&z_j-z_i\\
0&1
\end{pmatrix}\in\mathbf{SL}(2,\mathbb{C})\subset\mathbf{GL}(2,\mathbb{C}).$$
Here $z_i$ is a local coordinate of $U_i$ where $\{U_i\}$ is an open cover of  $\mathfrak{R}$. The affine embedding transformation is given by
$$\xymatrix{(w,1)\ar[r]^(0.3){G_{ij}}& (w+(z_{j}-z_{i}),1).}$$
A vector bundle of rank $2$ with the set of transition functions  of type
$\{G_{ij}(q)=\begin{pmatrix}
a_{ij}(q)&b_{ij}(q)\\
0&c_{ij}(q)\end{pmatrix}\}$ gives a $\mathbb{P}^1$-bundle and consequently it gives a ruled surface (see \cite{maru70} for details). We will let
$$\mathfrak{S}_1:=\mathbb{P}(\WW).$$

On the other hands, in \cite{kim11} we have constructed a ruled surface $\mathfrak{S}$ by elementary transformations and showed the followings: Let $K_\mathfrak{R}=\sum_{i=1}^{2g-2}q_i$ be a canonical divisor of $\mathfrak{R}$. In particular, let us assume all the $q_i$ are distinct throughout the paper unless otherwise specified. Let $p_i=(\infty,q_i)$ for $i=1,\dots,2g-2$ and $p'_i=([-1,1],q_i)$ for $i=1,\dots,2g-2$. Note that $p_i\in\mathbb{P}^1\times\mathfrak{R}$ and we denote $\infty=[1,0]$.

\begin{theorem}\label{elco1}\cite{kim11}

Let $\{(U_i,z_i)\}$ be an open cover with local coordinates $z_i$ of $\mathfrak{R}$. The transition function $G_{ij}\in\mathbf{PGL}(2,\mathbb{C})$ on $U_i\cap U_j$ of a projective bundle

$$\elm_{p_1}\circ\dots\circ \elm_{p_{2g-2}}\circ\elm_{p'_1}\circ\dots\circ \elm_{p'_{2g-2}} \Big(\mathbb{P}^1\times\mathfrak{R}\Big)$$
is given by
$$G_{ij}(q)=\Big[\begin{pmatrix}
1&(g_{ij}(q)-1)z_i(q)\\
0&1\end{pmatrix}\Big]\in\mathbf{PGL}(2,\mathbb{C})$$
Here $\{g_{ij}(q)=\frac{z_j(q)}{z_i(q)}\}$ is the set of the transition functions of a line bundle $\mathcal{O}_\mathfrak{R}(K_\mathfrak{R})$.
\end{theorem}

We will denote
$$\mathfrak{S}:=\elm_{p_1}\circ\dots\circ \elm_{p_{2g-2}}\circ\elm_{p'_1}\circ\dots\circ \elm_{p'_{2g-2}} \Big(\mathbb{P}^1\times\mathfrak{R}\Big).$$
From Theorem~\ref{elco1}, we can conclude that
$$\mathfrak{S}_1=\mathfrak{S}.$$
We will remark a couple of facts concerning the surface $\mathfrak{S}$ for reader's convenience. For more details, we refer \cite{kim11}: Let $C_0$ be a section of $\pi_\mathfrak{S}:\mathfrak{S}\to \mathfrak{R}$ corresponding to a trivial sub-bundle. It is obvious that $\HH^0(\mathfrak{R},\mathcal{W})\ne0$ where $\mathcal{W}$ is the sheaf of $\WW$. Moreover, $\HH^0(\mathfrak{R},\mathcal{W}\otimes\mathcal{L})=0$ for any line bundle $\mathcal{L}$ with negative degree. Hence, $\mathcal{W}$ is \emph{normalized} (see \emph{p.373} in \cite{har77}). Consequently,
$$C_0.C_0=qf.qf=0\text{ and }C_0.qf=1\text{ where }$$
$qf$ is a divisor which is the fiber $\pi^{-1}_{\mathfrak{S}}(q)$ of $\pi_\mathfrak{S}:\mathfrak{S}\to\mathfrak{R}$. Moreover, from \emph{p.373} in \cite{har77} we see that the canonical divisor
$$K_\mathfrak{S}\sim -2C_0+K_\mathfrak{R}f$$
where $K_\mathfrak{R}$ is a canonical divisor of $\mathfrak{R}$.

From these two different constructions of $\mathfrak{S}$, we can have one important observation: The ruled surface $\mathfrak{S}$ acquires a natural \emph{involution} $\tau$ when $\mathfrak{R}$ is a hyperelliptic curve: Let $\tau$ be a hyperelliptic involution of a given hyperelliptic curve $(\mathfrak{R},\tau)$. Then the hyperelliptic involution $\tau$ will induce an involution $\tau^\ast$ on $\Jac(\mathfrak{R})$. It is not hard to see that the natural involution $\tau'$ on $(\mathbb{C}^\ast)^{2g}$ given by $(z_1,\dots,z_{2g})\mapsto(-z_1,\dots,-z_{2g})$ induces $\tau^\ast$ by the affine bundle structure of short exact sequence \eqref{grse1} of commutative groups. Hence, we will denote the natural involution on  $(\mathbb{C}^\ast)^{2g}$ by $\tau$ by abusing a notation. By restricting to $\Ab(\mathfrak{R})$ which is obviously preserved by $\tau^\ast$, this will give an involution $\tau$ on $\mathfrak{S}=\mathbb{P}(\WW)$.

From Theorem~\ref{elco1}, we see that this involution $\tau$ on
$$\elm_{p_1}\circ\dots\circ \elm_{p_{2g-2}}\circ\elm_{p'_1}\circ\dots\circ \elm_{p'_{2g-2}} \Big(\mathbb{P}^1\times\mathfrak{R}\Big)$$
fixes $C_0$ and $K_\mathfrak{R}f$ set-theoretically and is nothing but the induced involution from an involution $(T,z)\mapsto (-T,\tau(z))$ on $\mathbb{P}^1\times\mathfrak{R}$ by elementary transformations. This involution $\tau$ will play a major role to study the finiteness of Hitchin hyperelliptic tangential covers over a given hyperelliptic curve.

\begin{remark}
From what we have shown, we may see that the construction is the generalization of the construction of a surface $S$ in \cite{tv90} which is a projectivization of rank $2$ bundle $\WW$ over an elliptic curve $X$. The bundle $\WW$ is an universal embedding bundle of a principal affine-bundle $\Delta$ over $X$:
$$0\to\mathbb{G}_a\to\Delta\to X\to 0.$$
\end{remark}

\section{Hitchin hyperelliptic Tangential Covers}\label{tansec3}
Throughout this paper, we assume that $q_i$ $i=1,\dots,2g-2$ in the canonical divisor $K_\mathfrak{R}=\sum_{i=1}^{2g-2}q_i$ are distinct and the ramification divisor of the involution $\tau$ of $(\mathfrak{R},\tau)$ do not intersect to the canonical divisor.

The involution $\tau$ constructed in Section~\ref{unisec1} has $4g+4$ fixed points which consist of two points over each ramification point of a hyperelliptic curve $\mathfrak{R}$ in the ruled surface $\pi_\mathfrak{S}:\mathfrak{S}\to\mathfrak{R}$. Note that $2g+2$ are on $C_0$ and $2g+2$ are on $\Delta$. We denote $s_i$ by those points on $C_0$ and $r_i$ by points in $\Delta$ for $i=1,\dots,2g+2$.

\begin{definition}
Let $(\mathfrak{R},\tau)$ be a hyperelliptic curve. We say that a Hitchin tangential cover $\pi:\widehat{\mathfrak{R}}\to \mathfrak{R}$ is {\em symmetric} if there is an involution $\tau'$ of $\widehat{\mathfrak{R}}$ under which $\pi^{-1}(K_\mathfrak{R})$ is invariant set-theoretically and which induces the canonical involution $\tau$ of $\mathfrak{R}$, i.e., $\pi\circ\tau'=\tau\circ\pi$.
\end{definition}

\begin{lemma}
Any Hitchin hyperelliptic tangential cover over a hyperelliptic curve is symmetric.
\end{lemma}

\begin{proof}
Let $(\mathfrak{R},\tau)$ be a hyperelliptic curve and $\pi:\widehat{\mathfrak{R}}\to \mathfrak{R}$ be a Hitchin hyperelliptic tangential cover. 
Clearly, we may take a canonical divisor $K_\mathfrak{R}$ set-theoretically invariant under $\tau$. The involution $\tau$ induces an involution $\tau'$ on $\widehat{\mathfrak{R}}$ and $\pi(\widehat{\mathfrak{R}}/\tau')=\mathbb{P}^1$. Moreover, $\tau'$ is a hyperelliptic involution of $\widehat{\mathfrak{R}}$. Consequently, we have a commutative diagram
$$\xymatrix{
\ar[d]_{\pi}\ar[r]\widehat{\mathfrak{R}}         &\ar[d]^{\pi_\tau} \widehat{\mathfrak{R}}/\tau'\cong \mathbb{P}^1\\
\ar[r]\mathfrak{R}                              & \mathfrak{R}/\tau\cong \mathbb{P}^1. 
}$$
\end{proof}

Let $\bl:\widetilde{\mathfrak{S}}\to\mathfrak{S}$ be the blow-up of $\mathfrak{S}$ at $\sigma(K_\mathfrak{R})$ where $\sigma:\mathfrak{R}\to C_0$ is the section. Since the canonical divisor $K_\mathfrak{R}$ is preserved by $\tau$, it is not hard to see that the involution $\tau$ on $\mathfrak{S}$ extends to $\widetilde{\tau}$ on $\widetilde{\mathfrak{S}}$. Let  $\bl':\widetilde{\mathfrak{S}}^\perp\to\widetilde{\mathfrak{S}}$ be the blow-up of $\widetilde{\mathfrak{S}}$ at $s_i$ and $r_i$ for $i=1,\dots,2g+2$ and consider the following diagram:

\begin{equation}\label{xyd1}
\xymatrix{
\ar[d]_{\bl'}\ar[r]^\varphi\widetilde{\mathfrak{S}}^\perp          &\ar[d]  \mathfrak{S}^\dagger\\
\ar[r]\ar[d]_{\bl}\widetilde{\mathfrak{S}}                              & \ar[d]\widetilde{\mathfrak{S}}/\widetilde{\tau} \\
\ar[r]\mathfrak{S}                               & \mathfrak{S}/\tau.
}
\end{equation}
Note that $\mathfrak{S}^\dagger$ is a minimal desingularization of $\widetilde{\mathfrak{S}}/\widetilde{\tau}$ and $\varphi$ is a covering of degree $2$ with the pullback $\sum_{i=1}^{2g+2}(s_{i}^{\perp}+r_{i}^{\perp})$ of the ramification divisor on $\widetilde{\mathfrak{S}}^\perp$ where $s_{i}^{\perp}$ and $r_{i}^{\perp}$ are exceptional curves for $i=1,\dots,2g+2$ after the blow-up of points $s_i$ and $r_i$ respectively.

\begin{lemma}\label{lemma7}
Let $\pi:\widehat{\mathfrak{R}}\to \mathfrak{R}$ be a Hitchin hyperelliptic tangential cover. There exists a unique $\widetilde{\tau}$-equivariant morphism $\widetilde{\iota}^\perp:\widehat{\mathfrak{R}}\to \widetilde{\mathfrak{S}}^\perp$.
\end{lemma}

\begin{proof}
From \cite{kim11}, there is a morphism $\widetilde{\iota}:\widehat{\mathfrak{R}}\to \widetilde{\mathfrak{S}}$ of degree $1$. The morphism $\widetilde{\iota}$ commutes with $\widetilde{\tau}$, and by passing to the quotient, we have a morphism $\widetilde{\iota}/\widetilde{\tau}:\widehat{\mathfrak{R}}/\tau\to\widetilde{\mathfrak{S}}/\widetilde{\tau}$. Since $\widehat{\mathfrak{R}}$ is hyperelliptic, we have $\widehat{\mathfrak{R}}/\tau\cong\mathbb{P}^1$ and in particular $\widehat{\mathfrak{R}}/\tau$ is smooth. Hence the morphism $\widetilde{\iota}/\widetilde{\tau}$ covers a morphism $(\widetilde{\iota}/\widetilde{\tau})^\dagger:\widehat{\mathfrak{R}}/\tau\to\mathfrak{S}^\dagger$. Since diagram~\eqref{xyd1} is commutative, there is a morphism $\widetilde{\iota}^\perp: \widehat{\mathfrak{R}} \to\widetilde{\mathfrak{S}}^\perp$ covered by $\widetilde{\iota}$. The uniqueness of $\widetilde{\iota}^\perp$ follows from the fact that $\widetilde{\iota}^\perp$ coincides with $\widetilde{\iota}$ on a dense open set. The equivariance with respect to $\widetilde{\tau}$ follows from the uniqueness.
\end{proof}

From Lemma~\ref{lemma7}, we may define
\begin{equation}\label{eqrho}
\varrho(\widehat{\mathfrak{R}}):=\varphi\circ\widetilde{\iota}^\perp(\widehat{\mathfrak{R}}).
\end{equation}
Let $\widetilde{\iota}:\widehat{\mathfrak{R}}\to\widetilde{\mathfrak{S}}$ and $\mu_i(\widehat{\mathfrak{R}})=\mult_{r_i}\iota(\widehat{\mathfrak{R}})$ for $i=1,\dots,2g+2$ and
$$\mu(\widehat{\mathfrak{R}})=(\mu_i(\widehat{\mathfrak{R}})_{1\leq i\leq 2g+2})\in\mathbb{N}^{2g+2}.$$
We say that $\mu(\widehat{\mathfrak{R}})$ is \emph{the type of cover} $\widehat{\mathfrak{R}}$ and $\mu(\widehat{\mathfrak{R}})$ \emph{is adapted to} $n$ if
$$\mu_i(\widehat{\mathfrak{R}})\equiv n\mod 2\text{ for }i=1,\dots,2g+2.$$
Also we let
\begin{equation}\label{eq4}
L_{n,n,n}:=\bl^\ast(nC_0+nK_\mathfrak{R}f)-n\sum_{i=1}^{2g-2}E_i.
\end{equation}
We know that $L_{n,n,n}$ is smooth (see \cite{kim11}) and the strict transformation of $nC_0+nK_\mathfrak{R}f$.

\begin{lemma}\label{lemma2}
For all $n>2$ and all $\mu$ adapted to $n$, there exists a unique divisor $\lambda(n,\mu)\in\pic(\mathfrak{S}^\dagger)$ such that

\begin{equation}\label{eq2}
\varphi^\ast(\lambda(n,\mu))=\bl'^\ast(L_{n,n,n})-\sum_{i=1}^{2g+2}\mu_i r_{i}^{\perp}.
\end{equation}
\end{lemma}

\begin{proof}
First let us prove that $\mathfrak{S}^\dagger$ is a rational surface: From diagram~\eqref{xyd1}, we have two injections
$$\HH^0(\mathfrak{S}^\dagger,\Omega^{1}_{\mathfrak{S}^\dagger})\hookrightarrow\HH^0(\widetilde{\mathfrak{S}}^\perp,\Omega^{1}_{\widetilde{\mathfrak{S}}^\perp})^\tau\text{ and }\HH^0(\mathfrak{S}^\dagger,(\Omega^{2}_{\mathfrak{S}^\dagger})^2)\hookrightarrow\HH^0(\widetilde{\mathfrak{S}}^\perp,(\Omega^{2}_{\widetilde{\mathfrak{S}}^\perp})^2).$$
 It is clear that $\widetilde{\mathfrak{S}}^\perp$ is birational to $\mathfrak{R}\times\mathbb{P}^1$ and the involution $\tau$ acts by $-1$ on
$$\HH^0(\mathfrak{R}\times\mathbb{P}^1,\Omega^{1}_{\mathfrak{R}\times\mathbb{P}^1})\cong\HH^0(\mathfrak{R},\Omega^{1}_{\mathfrak{R}}).$$
Hence, $\HH^0(\widetilde{\mathfrak{S}}^\perp,\Omega^{1}_{\widetilde{\mathfrak{S}}^\perp})^\tau\cong\HH^0(\mathfrak{R},\Omega^{1}_{\mathfrak{R}})^\tau=0$. Moreover, we have
$$\HH^0(\widetilde{\mathfrak{S}}^\perp,(\Omega^{2}_{\widetilde{\mathfrak{S}}^\perp})^2)\cong\HH^0(\mathfrak{R},(\Omega^{1}_{\mathfrak{R}})^2)\otimes\HH^0(\mathbb{P}^1,(\Omega^{1}_{\mathbb{P}^1})^2)=0.$$
Hence, by the Castelnuovo-Enriques criterion, we conclude that  $\mathfrak{S}^\dagger$ is a rational surface. Form this, we see that $\pic(\mathfrak{S}^\dagger)$ has no torsion. Consequently, $\varphi^\ast:\pic(\mathfrak{S}^\dagger)\to\pic(\widetilde{\mathfrak{S}}^\perp)$ is injective. Hence we prove the uniqueness of $\lambda(n,\mu)$ if it exists. The existence will follow from an explicit construction. We may let for $n$ even
$$\lambda(n,\mu)=nC_{0}^{\dagger}+(nK_\mathfrak{R}f)^\dagger+\frac{n}{2}\sum_{i=1}^{2g+2}s_{i}^{\dagger}-\sum_{i=1}^{2g+2}\frac{\mu_i}{2} r_{i}^{\dagger}.$$
Let $C_1$ be a unique symmetric divisor on $\mathfrak{S}$ which is isomorphic to the hyperelliptic curve $\mathfrak{R}$. Then for $n$ odd we may let 
$$\lambda(n,\mu)=C_{1}^{\dagger}+(n-1)C_{0}^{\dagger}+((n-1)K_\mathfrak{R}f)^\dagger+\frac{n-1}{2}\sum_{i=1}^{2g+2}s_{i}^{\dagger}-\sum_{i=1}^{2g+2}\frac{\mu_i-1}{2} r_{i}^{\dagger}.$$
It is not hard to see that

$$\begin{aligned}
\varphi^\ast(s_{i}^{\dagger})&=2s_{i}^{\perp},\varphi^\ast(r_{i}^{\dagger})=2r_{i}^{\perp}\\
\varphi^\ast(nC_{0}^{\dagger}+(nK_\mathfrak{R}f)^\dagger)&=\bl'^\ast\circ\bl^\ast(nC_0+nK_\mathfrak{R}f)-n\sum_{i=1}^{2g+2}s_{i}^{\perp}-n\sum_{i=1}^{2g-2}E_i\\
\varphi^\ast(C_{1}^{\dagger})&\equiv\bl'^\ast\circ\bl^\ast(C_0+K_\mathfrak{R}f)-\sum_{i=1}^{2g+2}(s_{i}^{\perp}+r_{i}^{\perp})-\sum_{i=1}^{2g-2}E_i.\\
\end{aligned}$$
Note that $\equiv$ in the above expression means a linear equivalence. Using the above, it is easy to see that
$$\varphi^\ast(\lambda(n,\mu))=\bl'^\ast(L_{n,n,n})-\sum_{i=1}^{2g+2}\mu_i r_{i}^{\perp}.$$
\end{proof}

Let $\mathscr{HHT}(n,g,\mathfrak{S},\mu)$ be the moduli space of Hitchin hyperelliptic tangential covers $\pi:\widehat{\mathfrak{R}}\to \mathfrak{R}$ of type of cover $\mu$ where the genus of $\mathfrak{R}$ is $g$ and $\deg(\pi)=n$. Note that an element of  $\mathscr{HHT}(n,g,\mathfrak{S},\mu)$ is a birational class of $\widehat{\mathfrak{R}}$.

\begin{theorem}\label{th11}
There is a one-to-one correspondence between  $\mathscr{HHT}(n,g,\mathfrak{S},\mu)$ and $|\lambda(n,\mu)|$ defined by
$$\widehat{\mathfrak{R}}\mapsto \varrho(\widehat{\mathfrak{R}}).$$
In particular, $\lambda(n,\mu)$ is a rational curve.
\end{theorem}

\begin{proof}

From \eqref{eqrho}, we have defined $\varrho(\widehat{\mathfrak{R}}):=\varphi\circ\widetilde{\iota}^\perp(\widehat{\mathfrak{R}})$ which is a rational curve. Since $\widetilde{\iota}(\widehat{\mathfrak{R}})\in|L_{n,n,n}|$, we have the strict transformation of $\widetilde{\iota}(\widehat{\mathfrak{R}})$ by Lemma~\ref{lemma7};
$$\widetilde{\iota}^\perp(\widehat{\mathfrak{R}})\in|\bl'^\ast(L_{n,n,n})-\sum_{i=1}^{2g+2}\mu_i r_{i}^{\perp}|.$$
Hence, we have by Lemma~\ref{lemma2}
$$\widetilde{\iota}^\perp(\widehat{\mathfrak{R}})=\varphi^\ast(\lambda(n,\mu)).$$
Again by the uniqueness of $\lambda(n,\mu)$ in Lemma~\ref{lemma2}, we have $\varrho(\widehat{\mathfrak{R}})\in|\lambda(n,\mu)|$.

Surjectivity of $\varrho$: For a rational curve $C\in|\lambda(n,\mu)|$ with normalization $C^\dagger$,
$$\Gamma=C^\dagger\times_{\mathfrak{S}^\dagger}\widetilde{\mathfrak{S}}^\perp.$$
Then we have a map $\widetilde{\iota}^\perp:\Gamma\to\widetilde{\mathfrak{S}}^\perp$. Define
$$\pi=\pi_\mathfrak{S}\circ\bl\circ\bl'\circ\widetilde{\iota}^\perp.$$
It is not hard to see that $\pi:\Gamma\to\mathfrak{R}$ is a Hitchin hyperelliptic tangential cover.

Injectivity of $\varrho$: Let $\varrho(\widehat{\mathfrak{R}})^\dagger$ be the normalization of $\varrho(\widehat{\mathfrak{R}})$ where $[\widehat{\mathfrak{R}}]\in\mathscr{HHT}(n,g,\mathfrak{S},\mu)$ where  $[\widehat{\mathfrak{R}}]$ is a birational class. As in the proof of the surjectivity, we may construct a Hitchin hyperelliptic tangential cover
$$\varrho(\widehat{\mathfrak{R}})^\dagger\times_{\mathfrak{S}^\dagger}\widetilde{\mathfrak{S}}^\perp.$$
In order to show the injectivity, it suffices to show that there is $\Gamma\in[\widehat{\mathfrak{R}}]$ such that
$$\Gamma=\varrho(\widehat{\mathfrak{R}})^\dagger\times_{\mathfrak{S}^\dagger}\widetilde{\mathfrak{S}}^\perp.$$
Clearly, the image of $\Gamma$ in $\widetilde{\mathfrak{S}}$ is the image of $\widehat{\mathfrak{R}}$. Hence, we have a birational map $\widehat{\mathfrak{R}}\to\Gamma$. Hence, $\Gamma$ is uniquely determined by $\varrho(\widehat{\mathfrak{R}})^\dagger$.
\end{proof}

\section{Finiteness of Hitchin hyperelliptic covers}\label{secfin4}

The canonical divisor of $\widetilde{\mathfrak{S}}$ is given by
\begin{equation}\label{eq3}
K_ {\widetilde{\mathfrak{S}}}=\bl^\ast(K_\mathfrak{S})+\sum_{i=1}^{2g-2}E_i=\bl^\ast(-2C_0+K_\mathfrak{R}f)+\sum_{i=1}^{2g-2}E_i.
\end{equation}
Since $\varphi$ is ramified along $\sum_{i=1}^{2g+2}(s_{i}^{\perp}+r_{i}^{\perp})$, we have
\begin{equation}\label{eq1}
\varphi^\ast(K_{\mathfrak{S}^\dagger})=K_{\widetilde{\mathfrak{S}}^\perp}-\sum_{i=1}^{2g+2}(s_{i}^{\perp}+r_{i}^{\perp})=\bl'^\ast(K_{\widetilde{\mathfrak{S}} }).
\end{equation}
Note the facts that $E_i.E_j=s_{i}^{\perp}.s_{j}^{\perp}=r_{i}^{\perp}.r_{j}^{\perp}=-\delta_{ij}$, $s_{i}^{\perp}.r_{j}^{\perp}=0$, $\bl^\ast(D_1).E_i=0$, and $\bl^\ast(D_1).\bl^\ast(D_2)=D_1.D_2$ for $D_1,D_2\in\pic(\mathfrak{S})$ where $\delta_{ij}$ is the Kronecker symbol.

\begin{lemma}\label{simlemma}
$$\lambda(n,\mu).K_{\mathfrak{S}^\dagger}=0.$$
\end{lemma}

\begin{proof}
Note that $\varphi_{\ast}\circ\varphi^\ast(\lambda(n,\mu))=2\lambda(n,\mu)$. Consequently,
$$\begin{aligned}
2\lambda(n,\mu).K_{\mathfrak{S}^\dagger}&=\varphi_{\ast}\circ\varphi^\ast(\lambda(n,\mu)).K_{\mathfrak{S}^\dagger}=\varphi^\ast(\lambda(n,\mu)).\varphi^\ast(K_{\mathfrak{S}^\dagger})\\
&\overset{\eqref{eq1},\text{Lemma}~\ref{lemma2}}{=}L_{n,n,n}^{\perp}.\bl'^\ast(K_{\widetilde{\mathfrak{S}} })\\
&\overset{\eqref{eq3},\eqref{eq2}}{=}\Big(\bl'^\ast(L_{n,n,n})-\sum_{i=1}^{2g+2}\mu_i r_{i}^{\perp}\Big).\bl'^\ast\Big(\bl^\ast(-2C_0+K_\mathfrak{R}f)+\sum_{i=1}^{2g-2}E_i\Big)\\
&\overset{\eqref{eq4}}{=}\Big(\bl^\ast(nC_0+nK_\mathfrak{R}f)-n\sum_{i=1}^{2g-2}E_i\Big).\Big(\bl^\ast(-2C_0+K_\mathfrak{R}f)+\sum_{i=1}^{2g-2}E_i\Big)\\
&=n(2g-2)-2n(2g-2)+n(2g-2)=0.
\end{aligned}$$
\end{proof}

\begin{lemma}\label{simle2}
Let $\mu^{(2)}:=\sum_{i=1}^{2g+2}\mu_{i}^{2}$. We have
$$\lambda(n,\mu)^2=\frac{n^2(2g-2)-\mu^{(2)}}{2}.$$
\end{lemma}

\begin{proof}
The proof is also a simple calculation similar to the proof of Lemma~\ref{simlemma}:
$$\begin{aligned}
2\lambda(n,\mu)^2&=\varphi_{\ast}\circ\varphi^\ast(\lambda(n,\mu)).\lambda(n,\mu)=\varphi^\ast(\lambda(n,\mu)).\varphi^\ast(\lambda(n,\mu))\\
&=(L_{n,n,n}^{\perp})^2=\Big(\bl'^\ast(L_{n,n,n})-\sum_{i=1}^{2g+2}\mu_i r_{i}^{\perp}\Big)^2\\
&=\Big(\bl'^\ast(\bl^\ast(nC_0+nK_\mathfrak{R}f)-n\sum_{i=1}^{2g-2}E_i)-\sum_{i=1}^{2g+2}\mu_i r_{i}^{\perp}\Big)^2\\
&=2n^2(2g-2)-n^2(2g-2)-(\sum_{i=1}^{2g+2}\mu_{i}^{2})\\
&=n^2(2g-2)-\mu^{(2)}.
\end{aligned}$$

\end{proof}

The following lemma is essentially the same as one in \cite{tv90} except one assumption. We reproduce the almost same proof here.

\begin{lemma}\label{lemll}\cite{tv90}
Let $\Sigma$ be an analytic complex surface, $K_\Sigma$ a canonical divisor, and $f:\mathbb{P}^1\to \Sigma$ a morphism. We suppose that $f(\mathbb{P}^1).K_\Sigma=0$. Then for all germs of analytic deformation $t\mapsto f_t$ of $f=f_0$, we have $f_t(\mathbb{P}^1)=f_0(\mathbb{P}^1)$ for all $t$ in the neighborhood of $0$.
\end{lemma}

\begin{proof}
Let $\TT_\Sigma$ and $\TT_{\mathbb{P}^1}$ be the tangent bundles respectively. Using $f_t$, we have an exact sequence 
$$0\to \TT_{\mathbb{P}^1}\to f_{t}^{\ast} \TT_\Sigma\to N_t\to0$$
where $N_t$ is a normal sheaf on $\mathbb{P}^1$. The hypothesis implies that the degree of $f_{t}^{\ast} T_\Sigma$ is equal to $0$. Since the degree of $\TT_{\mathbb{P}^1}$ is $2$, for all $t$ in the neighborhood of zero, $N_t$ is a direct sum of locally free sheaf of rank $1$ with strictly negative degree and a torsion sheaf. Therefore, the image on $N_t$ of a section $\partial f/\partial t$ which is a section of $f_{t}^{\ast}\TT_\Sigma$ is zero except a finite number of points at most. Hence $f_t=f_0$.
\end{proof}

\begin{theorem}\label{th44}
We have
$$n^2(2g-2)+4-\mu^{(2)}\geq0.$$
Consequently, there are only finitely many possible
$$\mu(\widehat{\mathfrak{R}})=(\mu_i(\widehat{\mathfrak{R}})_{1\leq i\leq 2g+2})\in\mathbb{N}^{2g+2}.$$
\end{theorem}

\begin{proof}
Let $\varrho(\widehat{\mathfrak{R}})\in|\lambda(n,\mu)|$. Using Lemma~\ref{simlemma} and Lemma~\ref{simle2} combining with the adjunction formula, we have
$$\begin{aligned}
g(\varrho(\widehat{\mathfrak{R}}))&=1+\frac{\varrho(\widehat{\mathfrak{R}}).\varrho(\widehat{\mathfrak{R}})+K_{\mathfrak{S}^\dagger}.\varrho(\widehat{\mathfrak{R}})}{2}\\
&=1+\frac{n^2(2g-2)-\mu^{(2)}}{4}.
\end{aligned}$$
Since the genus $g(\varrho(\widehat{\mathfrak{R}}))$ is $\geq0$, we have the desired result.
\end{proof}

Finally, we have the main result:

\begin{corollary}\label{th45}
There are only finitely many Hitchin hyperelliptic tangential covers over a hyperelliptic curve.
\end{corollary}

\begin{proof}
From Lemma~\ref{lemll}, we have $\dim_\mathbb{C}|\lambda(n,\mu)|=0$ and from Theorem~\ref{th44}, there are finitely many $\mu$. Finally, combining the above with Theorem~\ref{th11}, 
the claim follows.
\end{proof}

Combining Corollary~\ref{th45} with \eqref{dimeq}, we have the following corollary:

\begin{corollary}
Let $n>2$. The moduli space of Hitchin hyperelliptic spectral covers over a hyperelliptic curve has at most co-dimension of $g+1$ in the moduli space of all Hitchin spectral covers if there exists a Hitchin hyperelliptic tangential covers.

\end{corollary}

\end{document}